\documentclass[12pt]{amsart}

\usepackage{amsmath,amscd}
\usepackage{amssymb}
\usepackage{amsthm}

\usepackage{mathrsfs}
\usepackage{hyperref}

\usepackage{calc}
\newenvironment{tehtratk}%
             {\begin{list}{\arabic{enumi}.}{\usecounter{enumi}%
              \setlength{\labelsep}{0.5em}%
              \settowidth{\labelwidth}{(\arabic{enumi})}%
              \setlength{\leftmargin}{\labelwidth+\labelsep}}}%
             {\end{list}}

\newtheorem{Theorem}{Theorem}[section]

\newtheorem{Lemma}[Theorem]{Lemma}

\theoremstyle{definition}

\newtheorem{Remark}[Theorem]{Remark}

\numberwithin{equation}{section}

\newcommand{\mR}{\mathbb{R}}                    
\newcommand{\abs}[1]{\lvert #1 \rvert}          
\newcommand{\norm}[1]{\lVert #1 \rVert}         
\newcommand{\br}[1]{\langle #1 \rangle}         

\newcommand{\ol}[1]{\overline{#1}}

\newcommand{\mS}{\mathscr{S}}

\newcommand{\mF}{\mathscr{F}}

\newcommand{\eps}{\varepsilon}

\newcounter{sidenote}
\begin{document}

\title[Calder{\'o}n problem for fractional Schr\"odinger equation]{The Calder{\'o}n problem for the fractional Schr\"odinger equation} 

\author[T. Ghosh]{Tuhin Ghosh}
\address{Jockey Club Institute for Advanced Study, HKUST, Hong Kong}
\email{iasghosh@ust.hk}

\author[M. Salo]{Mikko Salo}
\address{Department of Mathematics and Statistics, University of Jyv\"askyl\"a}
\email{mikko.j.salo@jyu.fi}

\author[G. Uhlmann]{Gunther Uhlmann}
\address{Department of Mathematics, University of Washington / Jockey Club Institute for Advanced Study, HKUST, Hong Kong / Department of Mathematics and Statistics, University of Helsinki}
\email{gunther@math.washington.edu}




\begin{abstract}
We show global uniqueness in an inverse problem for the fractional Schr\"odinger equation: an unknown potential in a bounded domain is uniquely determined by exterior measurements of solutions. We also show global uniqueness in the partial data problem where measurements are taken in arbitrary open, possibly disjoint, subsets of the exterior. The results apply in any dimension $\geq 2$ and are based on a strong approximation property of the fractional equation that extends earlier work. This special feature of the nonlocal equation renders the analysis of related inverse problems radically different from the traditional Calder\'on problem.
\end{abstract}

\maketitle

\section{Introduction} \label{sec_introduction}

In this article we consider a nonlocal analogue of the inverse conductivity problem posed by Calder\'on \cite{Calderon}. In the standard Calder\'on problem, the objective is to determine the electrical conductivity of a medium from voltage and current measurements on its boundary. This problem is the mathematical model of Electrical Resistivity/Impedance Tomography in seismic, medical and industrial imaging. It serves as a model case for various inverse problems for elliptic equations, and has a rich mathematical theory with connections to many other questions. We refer to the survey \cite{Uhlmann_survey} for more details.

In mathematical terms, if $\Omega \subset \mR^n$ is a bounded open set with Lipschitz boundary (the medium of interest), after a standard reduction one often considers the Dirichlet problem for the Schr\"odinger equation 
\[
(-\Delta + q)u = 0 \text{ in $\Omega$}, \qquad u|_{\partial \Omega} = f
\]
where $q \in L^{\infty}(\Omega)$ and $0$ is not a Dirichlet eigenvalue for $-\Delta + q$ in $\Omega$. The boundary measurements are given by the Dirichlet-to-Neumann map (DN map) 
\[
\Lambda_q: H^{1/2}(\partial \Omega) \to H^{-1/2}(\partial \Omega),
\]
defined weakly in terms of the bilinear form for the equation. Here and below, we denote the standard $L^2$ based Sobolev spaces by $H^s$.

For more regular boundaries and functions $f$, the DN map is given by the normal derivative $\Lambda_q f = \partial_{\nu} u|_{\partial \Omega}$ where $u$ is the solution with boundary value $f$. The inverse problem is to determine the potential $q$ in $\Omega$ from the knowledge of the DN map $\Lambda_q$.

We will consider an inverse problem for a nonlocal analogue of the Schr\"odinger equation. In fact, our equation will be the fractional Schr\"odinger equation $((-\Delta)^s + q) u = 0$ in $\Omega$ where $0 < s < 1$. Here the fractional Laplacian is defined by 
\[
(-\Delta)^s u = \mF^{-1} \{ \abs{\xi}^{2s} \hat{u}(\xi) \}, \qquad u \in H^s(\mR^n),
\]
and $\hat{u} = \mF u$ is the Fourier transform of $u$. This operator is nonlocal (it does not preserve the support of $u$), and one natural way to set up the Dirichlet problem is to look for solutions $u \in H^s(\mR^n)$ satisfying 
\[
((-\Delta)^s + q)u = 0 \text{ in $\Omega$}, \qquad u|_{\Omega_e} = f
\]
where $f \in H^s(\Omega_e)$, and $\Omega_e$ is the exterior domain 
\[
\Omega_e = \mR^n \setminus \ol{\Omega}.
\]
We recall basic facts about weak solutions in Section \ref{sec_preliminaries}. In particular, there is a countable set of Dirichlet eigenvalues, and we will assume that $q$ is such that $0$ is not an eigenvalue, that is: 
\begin{equation} \label{dirichlet_uniqueness}
\left\{Ê\begin{array}{c} \text{if $u \in H^s(\mR^n)$ solves $((-\Delta)^s + q)u = 0$ in $\Omega$ and $u|_{\Omega_e} = 0$,} \\
\text{then $u \equiv 0$.} \end{array} \right.
\end{equation}
This holds e.g.\ if $q \geq 0$. Then there is a unique solution $u \in H^s(\mR^n)$ for any $f \in H^s(\Omega_e)$, and one may define an analogue of the DN map, 
\[
\Lambda_q: H^s(\Omega_e) \to H^s(\Omega_e)^*
\]
that maps $f$ to a nonlocal analogue of the Neumann boundary value of the solution $u$. (This discussion assumed that $\Omega$ is a bounded Lipschitz domain, see Section \ref{sec_preliminaries} for the case of general bounded open sets.)

We will define $\Lambda_q$ via the bilinear form associated with the fractional Dirichlet problem. There are other nonlocal Neumann operators that one could use, but by Theorem \ref{thm_main} any reasonable measurement operator would be determined by $\Lambda_q$ (we will verify this directly for the operator $\mathcal{N}_s$ in \cite{DipierroRosOtonValdinoci}). Again, if $\Omega$ has $C^{\infty}$ boundary and $q$ and $f$ are more regular, the DN map is more explicit and is given by 
\[
\Lambda_q: H^{s+\beta}(\Omega_e) \to H^{-s+\beta}(\Omega_e), \ \ \Lambda_q f = (-\Delta)^s u|_{\Omega_e}
\]
where $u$ is the solution of $((-\Delta)^s + q)u = 0$ in $\Omega$ with exterior value $f$, and $\max \{0, s-1/2\} < \beta < 1/2$ (such a $\beta$ exists since $0 < s < 1$). Heuristically, given an open set $W \subset \Omega_e$, one can interpret $\Lambda_q f|_W$ as measuring the cost required to maintain the exterior value $f$ in $W$.

The following theorem is the main result in this article. It solves the fractional Schr\"odinger inverse problem in any dimension $n \geq 2$, and also the partial data problem with exterior Dirichlet and Neumann measurements in arbitrary open (possibly disjoint) sets $W_1, W_2 \subset \Omega_e$.

\begin{Theorem} \label{thm_main}
Let $\Omega \subset \mR^n$, $n \geq 2$, be bounded open, let $0 < s < 1$, and let $q_1, q_2 \in L^{\infty}(\Omega)$ satisfy \eqref{dirichlet_uniqueness}. Let also $W_1, W_2 \subset \Omega_e$ be open. If the DN maps for the equations $((-\Delta)^s + q_j) u = 0$ in $\Omega$ satisfy 
\[
\Lambda_{q_1} f|_{W_2} = \Lambda_{q_2} f|_{W_2} \text{ for any $f \in C^{\infty}_c(W_1)$,}
\]
then $q_1 = q_2$ in $\Omega$.
\end{Theorem}

For the usual Schr\"odinger equation $(-\Delta+q)u = 0$ and the related DN map $\Lambda_q$ on the full boundary $\partial \Omega$, the corresponding result is due to \cite{SylvesterUhlmann} when $n \geq 3$ and to \cite{Bukhgeim} when $n=2$ for slightly more regular potentials; for the case of $L^p$ potentials see \cite{BIY} when $n=2$ and \cite{Chanillo, LavineNachman} when $n \geq 3$. The partial data problem of determining $q$ from the knowledge of $\Lambda_q f|_{\Gamma}$ for any $f$ supported in $\Gamma$, when $\Gamma$ is an arbitrary open subset of $\partial \Omega$, was solved in \cite{IUY} when $n=2$ for $q_j \in C^{2,\alpha}$. The corresponding result in dimensions $n \geq 3$ is open, but there are several partial results including \cite{KSU}, \cite{I}, \cite{KenigSalo}. The case of measurements on disjoint sets is even more difficult, and counterexamples may appear \cite{IUY_disjoint, DKN1, DKN2}. See the surveys \cite{IY_survey, KenigSalo_survey} for further references.

The proof of Theorem \ref{thm_main} begins by showing that if the two DN maps are equal, then (exactly as in the usual Schr\"odinger case) one has the integral identity 
\[
\int_{\Omega} (q_1-q_2) u_1 u_2 \,dx = 0
\]
for any $u_j \in H^s(\mR^n)$ that solve $((-\Delta)^s + q_j) u_j = 0$ in $\Omega$ and satisfy $\mathrm{supp}(u_j) \subset \overline{\Omega} \cup \overline{W}_j$. For the standard Schr\"odinger equation, one then typically uses special complex geometrical optics solutions $u_j$ to show that the products $\{Êu_1 u_2 \}$ form a complete set in $L^1(\Omega)$. See \cite{Uhlmann_survey} for an overview.

However, solutions of the fractional Schr\"odinger equation are much less rigid than those of the usual Schr\"odinger equation. The fractional equation enjoys stronger uniqueness and approximation properties, as demonstrated by the following theorems:

\begin{Theorem} \label{thm_uniqueness}
If $0 < s < 1$, if $u \in H^{-r}(\mR^n)$ for some $r \in \mR$, and if both $u$ and $(-\Delta)^s u$ vanish in some open set, then $u \equiv 0$.
\end{Theorem}

\begin{Theorem} \label{thm_approximation}
Let $\Omega \subset \mR^n$ be a bounded open set, and let $\Omega_1 \subset \mR^n$ be any open set with $\Omega \subset \Omega_1$ and $\mathrm{int}(\Omega_1 \setminus \Omega) \neq \emptyset$.
\begin{itemize}
\item[(a)] 
If $q \in L^{\infty}(\Omega)$ satisfies \eqref{dirichlet_uniqueness}, then any $f \in L^2(\Omega)$ can be approximated arbitrarily well in $L^2(\Omega)$ by functions $u|_{\Omega}$ where $u \in H^s(\mR^n)$ satisfy 
\[
((-\Delta)^s + q) u = 0 \text{ in $\Omega$}, \qquad \mathrm{supp}(u) \subset \ol{\Omega}_1.
\]
\item[(b)] 
If $\Omega$ has $C^{\infty}$ boundary, and if $q \in C^{\infty}_c(\Omega)$ satisfies \eqref{dirichlet_uniqueness}, then any $f \in C^{\infty}(\ol{\Omega})$ can be approximated arbitrarily well in $C^{\infty}(\ol{\Omega})$ by functions $d(x)^{-s} u|_{\Omega}$ where $u \in H^s(\mR^n)$ satisfy 
\[
((-\Delta)^s + q) u = 0 \text{ in $\Omega$}, \qquad \mathrm{supp}(u) \subset \ol{\Omega}_1.
\]
(Here $d$ is any function in $C^{\infty}(\ol{\Omega})$ with $d(x) = \mathrm{dist}(x,\partial \Omega)$ near $\Omega$ and $d > 0$ in $\Omega$. Also, $v_j \to v$ in $C^{\infty}(\ol{\Omega})$ means that $v_j \to v$ in $C^k(\ol{\Omega})$ for all $k \geq 0$.)
\end{itemize}
\end{Theorem}

Note that both of these properties fail for the usual Laplacian: if $u \in C^{\infty}_c(\mR^n)$ then both $u$ and $\Delta u$ vanish in a large set but $u$ can be nontrivial, and the set of harmonic functions in $L^2(\Omega)$ is a closed subspace of $L^2(\Omega)$ which is smaller than $L^2(\Omega)$.

Theorem \ref{thm_uniqueness} is classical \cite{Riesz} at least with stronger conditions on $u$, and even the strong unique continuation principle holds in this context \cite{FallFelli, Ruland, Yu}. For later applications we will give a robust proof using the Carleman estimates from \cite{Ruland} and the Caffarelli-Silvestre extension \cite{CaffarelliSilvestre}.

The following version of Theorem \ref{thm_approximation} has been proved in \cite{DipierroSavinValdinoci}, see also \cite{ DipierroSavinValdinoci_nonlocal}: given $f \in C^k(\ol{B}_1)$ and $\eps > 0$, there is $u \in H^s(\mR^n)$ with $(-\Delta)^s u = 0$ in $B_1$ and $\mathrm{supp}(u) \subset \ol{B}_{R}$ for some possibly large $R = R_{\eps,f} > 1$, so that 
\[
\norm{u-f}_{C^k(\ol{B}_1)} < \eps.
\]
Theorem \ref{thm_approximation} improves this by reducing the approximation property to the uniqueness property, Theorem \ref{thm_uniqueness}, using a Runge type argument \cite{Lax, Malgrange} and regularity for fractional Dirichlet problems \cite{Hormander_unpublished, Grubb}. In particular, this implies that the result of \cite{DipierroSavinValdinoci} is valid for any fixed $R > 1$. The strong approximation property replaces the method of complex geometrical optics in solving the inverse problem for the fractional Schr\"odinger equation.

The study of fractional and nonlocal operators is currently an active research field and the related literature is substantial. We only mention that operators of this type arise in problems involving anomalous diffusion and random processes with jumps, and they have applications in probability theory, physics, finance, and biology. See \cite{BucurValdinoci, RosOton} for further information and references.

The mathematical study of inverse problems for fractional equations goes back at least to \cite{CNYY}. By now there are a number of results, mostly for time-fractional models and including many numerical works. Here is an example of the rigorous results that are available \cite{SakamotoYamamoto}: in the time-fractional heat equation 
\[
\partial_t^{\alpha} u - \Delta u = 0 \text{ in $\Omega \times (0,T)$,} \qquad u|_{\partial \Omega \times (0,T)} = 0,
\]
where $0 < \alpha < 1$ and $\partial_t^{\alpha}$ is the Caputo derivative, $u(0)$ is determined by $u(T)$ in a mildly ill-posed way (for $\alpha=1$ this problem is severely ill-posed). In general, nonlocality may influence the nature of the inverse problem but there are several aspects to be taken into account. We refer to \cite{JinRundell} for a detailed discussion and many further references. We are not aware of any  previous rigorous works on multidimensional inverse problems for space-fractional equations.

Finally, we note that Theorem \ref{thm_main} is a global uniqueness result in the inverse problem for the fractional Schr\"odinger equation, both with full and partial data. This could be used as as starting point for the study of reconstruction algorithms, stability properties and numerical implementations for inverse problems for the fractional Schr\"odinger equation and other nonlocal models as well.

This paper is organized as follows. Section \ref{sec_introduction} is the introduction. In Section \ref{sec_preliminaries} we review weak solutions of fractional Dirichlet problems, and give a definition of the DN map. For the benefit of those readers who may not be familiar with fractional Laplacians, we give rather complete (though concise) proofs using the Fourier transform and distribution theory as presented in \cite{Hormander1}. Section \ref{sec_dnmap_description} gives alternative descriptions of the DN map. In Sections \ref{sec_uniqueness} and \ref{sec_approximation_ltwo} we prove Theorems \ref{thm_uniqueness} and \ref{thm_approximation}(a). The solution of the inverse problem, Theorem \ref{thm_main}, is given in Section \ref{sec_inverse_problem}. In Section \ref{sec_approximation_sobolev} we invoke the regularity theory for fractional Dirichlet problems in \cite{Grubb} and prove Theorem \ref{thm_approximation}(b).

\subsection*{Acknowledgements}

M.S.\ was partly supported by the Academy of Finland (Centre of Excellence in Inverse Problems Research) and by an ERC Starting Grant (no 307023). G.U.\ was partly supported by NSF, a Si-Yuan Professorship at IAS,
HKUST, and a FiDiPro at U.\ Helsinki. The authors thank Gerd Grubb for several remarks that considerably improved Section \ref{sec_approximation_sobolev}, and Zhen-Qing Chen for helpful discussions.

\section{Fractional Laplacian} \label{sec_preliminaries}

In this section we review some basic facts about Dirichlet problems for the fractional Laplacian, see e.g.\ \cite{HohJacob, FelsingerKassmannVoigt, Grubb, RosOton}. For simplicity, we will assume most functions to be real valued in this paper.

\subsection{Sobolev spaces}


We first establish the notation for Sobolev type spaces. We write $H^s(\mR^n) = W^{s,2}(\mR^n)$ for the standard $L^2$ based Sobolev space with norm 
\[
\norm{u}_{H^s(\mR^n)} = \norm{\br{D}^s u}_{L^2(\mR^n)}
\]
where $\br{\xi} = (1+\abs{\xi}^2)^{1/2}$, and the notation $m(D) u = \mF^{-1} \{ m(\xi) \hat{u}(\xi) \}$ is used for Fourier multipliers when $m \in C^{\infty}(\mR^n)$ is polynomially bounded together with its derivatives. Our notation for the Fourier transform is 
\[
\hat{u}(\xi) = \mF u(\xi) = \int_{\mR^n} e^{-ix \cdot \xi} u(x) \,dx.
\]
If $U \subset \mR^n$ is an open set (not necessarily bounded), define the spaces (we follow the notation of \cite{McLean}) 
\begin{align*}
H^s(U) &= \{ u|_U \,;\, u \in H^s(\mR^n) \}, \\
\widetilde{H}^s(U) &= \text{closure of $C^{\infty}_c(U)$ in $H^s(\mR^n)$}, \\
H^s_0(U) &= \text{closure of $C^{\infty}_c(U)$ in $H^s(U)$}.
\end{align*}
We equip $H^s(U)$ with the quotient norm $\norm{u}_{H^s(U)} = \inf \{ \norm{w}_{H^s} \,;\, w \in H^s(\mR^n), \ w|_U = u \}$. Also, if $F \subset \mR^n$ is a closed set, we define 
\[
H^s_F = H^s_F(\mR^n) = \{ u \in H^s(\mR^n) \,;\, \mathrm{supp}(u) \subset F \}.
\]

We say that an open set $U \subset \mR^n$ is a Lipschitz domain if its boundary $\partial U$ is compact and if locally near each boundary point $U$ can be represented as the set above the graph of a Lipschitz function. Thus $U$ could be a bounded Lipschitz domain, or $U$ could be $\mR^n \setminus \overline{\Omega}$ where $\Omega$ is a bounded Lipschitz domain. If $U$ is a Lipschitz domain, then (with natural identifications, see \cite{McLean}, \cite{Triebel_Lipschitz}) 
\begin{gather*}
\widetilde{H}^s(U) = H^s_{\overline{U}}(\mR^n), \quad s \in \mR, \\
H^s_{\overline{U}}(\mR^n)^* = H^{-s}(U) \text{ and } H^s(U)^* = H^{-s}_{\overline{U}}(\mR^n), \quad s \in \mR, \\
H^s(U) = H^s_{\overline{U}}(\mR^n) = H^s_0(U), \quad -1/2 < s < 1/2.
\end{gather*}

\subsection{Fractional Laplacian}

Let $a > -n/2$ and consider the fractional Laplacian in $\mR^n$, 
\[
(-\Delta)^a u = \mF^{-1} \{ \abs{\xi}^{2a} \hat{u}(\xi) \}, \qquad u \in \mS,
\]
where $\mS$ denotes Schwartz space in $\mR^n$. If $\psi \in C^{\infty}_c(\mR^n)$ with $\psi=1$ near $0$, splitting $\abs{\xi}^{2a} = \psi(\xi) \abs{\xi}^{2a} + (1-\psi(\xi)) \abs{\xi}^{2a}$ and using the assumption $a > -n/2$ shows that $\abs{\xi}^{2a}$ is the sum of an $L^1$ function and a smooth function whose derivatives grow at most polynomially. Thus $(-\Delta)^a$ for $a > -n/2$ is a continuous map from $\mS$ to $L^{\infty}$.

There are many other definitions of the fractional Laplacian \cite{Kwasnicki}. For instance, if $0 < a < 1$ it is given by the principal value integral 
\[
(-\Delta)^a u(x) = c_{n,a} \ \mathrm{p.v.} \int_{\mR^n} \frac{u(x)-u(y)}{\abs{x-y}^{n+2a}} \,dy.
\]
We next extend $(-\Delta)^a$ to act on larger spaces. In particular, if $a \geq 0$, then $(-\Delta)^a$ will be well defined on $H^s(\mR^n)$ for any $s \in \mR$.

\begin{Lemma} \label{lemma_fractional_mapping_properties}
If $a \geq 0$, the fractional Laplacian extends as a bounded map 
\[
(-\Delta)^a: \ H^s(\mR^n) \to H^{s-2a}(\mR^n)
\]
whenever $s \in \mR$. If $-n/2 < a < 0$, the fractional Laplacian $(-\Delta)^a$ is the Riesz potential 
\[
(-\Delta)^a u = I_{2\abs{a}} u = \frac{c_{n,a}}{\abs{\,\cdot\,}^{n-2\abs{a}}} \ast u
\]
and it extends as a bounded map 
\[
(-\Delta)^a: L^p(\mR^n) \to L^{\frac{np}{n-2\abs{a} p}}(\mR^n), \qquad 1 < p < n/(2\abs{a}).
\]
\end{Lemma}
\begin{proof}
If $u \in \mS$, then 
\[
\norm{(-\Delta)^a u}_{H^{s-2a}} = \norm{\mF^{-1} \{ m(\xi) \br{\xi}^{s} \hat{u}(\xi) \} }_{L^2}
\]
where $m(\xi) = \br{\xi}^{-2a} \abs{\xi}^{2a}$ is bounded and hence a Fourier multiplier on $L^2$, showing that $\norm{(-\Delta)^a u}_{H^{s-2a}} \leq C \norm{u}_{H^{s}}$. The second statement is the Hardy-Littlewood-Sobolev inequality \cite[Theorem 4.5.3]{Hormander1}.
\end{proof}

\begin{Remark}
If $a \geq 0$, the fractional Laplacian also extends as a bounded map 
\begin{align*}
(-\Delta)^a:& \ W^{s,p}(\mR^n) \to W^{s-2a,p}(\mR^n), \\
(-\Delta)^a:& \ C^s_*(\mR^n) \to C^{s-2a}_*(\mR^n),
\end{align*}
whenever $s \in \mR$ and $1 < p < \infty$, where $W^{s,p}$ are the usual $L^p$ Sobolev (Bessel potential) spaces and $C^s_*$ are the Zygmund spaces (see \cite{Taylor3}). An even larger domain for $(-\Delta)^a$ is obtained as in \cite{Silvestre} by considering the test function space 
\[
\mS_a = \{ u \in C^{\infty}(\mR^n) \,;\, \br{\,\cdot\,}^{n+2a} \partial^{\alpha} u \in L^{\infty}(\mR^n) \text{ for any multi-index $\alpha$} \},
\]
equipped with the topology induced by the seminorms $\norm{\br{\,\cdot\,}^{n+2a} \partial^{\alpha} u}_{L^{\infty}}$. Then $(-\Delta)^a$ is continuous from $\mS$ to $\mS_a$ and extends to the dual 
\begin{gather*}
\mS_a' = \{ u \in \mS'(\mR^n) \,;\, u = \sum_{\abs{\alpha} \leq m} \partial^{\alpha} u_{\alpha} \text{ for some $m \geq 0$ and} \\ \hspace{180pt} \text{$u_{\alpha} \in \br{\,\cdot\,}^{n+2a} L^{\infty}(\mR^n)$} \}.
\end{gather*}
However, in this article it suffices to work with the spaces $H^s(\mR^n)$.
\end{Remark}

\subsection{Dirichlet problem}

Next we restrict our attention to nonlocal operators 
\[
(-\Delta)^s,  \ \ 0 < s < 1,
\]
and consider the solvability of the Dirichlet problem 
\begin{align*}
((-\Delta)^s + q)u &= F \quad \text{in $\Omega$}, \\
u &= f \quad \text{in $\Omega_e$},
\end{align*}
where, for a bounded open set $\Omega \subset \mR^n$, we denote the exterior domain by $\Omega_e = \mR^n \setminus \overline{\Omega}$. We also denote the restriction to $\Omega$ by 
\[
r_{\Omega} u = u|_{\Omega},
\]
and if $U \subset \mR^n$ is open and $u, v \in L^2(U)$ we write 
\[
(u,v)_U = \int_U u v \,dx.
\]

\begin{Lemma} \label{lemma_fractional_dirichlet_solvability}
Let $\Omega \subset \mR^n$ be a bounded open set, let $0 < s < 1$, and let $q \in L^{\infty}(\Omega)$. Let $B_q$ be the bilinear form defined for $v, w \in H^s(\mR^n)$ by 
\[
B_q(v,w) = ((-\Delta)^{s/2} v, (-\Delta)^{s/2} w)_{\mR^n} + (q r_{\Omega} v, r_{\Omega} w)_{\Omega}.
\]
\begin{tehtratk}
\item[{\rm (a)}]
There is a countable set $\Sigma = \{ \lambda_j \}_{j=1}^{\infty} \subset \mR$, $\lambda_1 \leq \lambda_2 \leq \cdots \to \infty$, with the following property: if $\lambda \in \mR \setminus \Sigma$, then for any $F \in (\widetilde{H}^s(\Omega))^*$ and $f \in H^s(\mR^n)$ there is a unique $u \in H^s(\mR^n)$ satisfying 
\[
B_q(u,w) - \lambda(u,w)_{\mR^n}= F(w) \ \text{ for $w \in \widetilde{H}^s(\Omega)$}, \quad u - f \in \widetilde{H}^s(\Omega).
\]
One has the norm estimate 
\[
\norm{u}_{H^s(\mR^n)} \leq C( \norm{F}_{(\widetilde{H}^s(\Omega))^*} + \norm{f}_{H^s(\mR^n)})
\]
with $C$ independent of $F$ and $f$.

\item[{\rm (b)}]
The function $u$ in (a) is also the unique $u \in H^s(\mR^n)$ satisfying  
\[
r_{\Omega} ((-\Delta)^s + q - \lambda) u = F \text{ in the sense of distributions in $\Omega$}
\]
and $u-f \in \widetilde{H}^s(\Omega)$. 

\item[{\rm (c)}]
One has $0 \notin \Sigma$ if \eqref{dirichlet_uniqueness} holds. If $q \geq 0$, then one has $\Sigma \subset (0,\infty)$ and \eqref{dirichlet_uniqueness} always holds.
\end{tehtratk}
\end{Lemma}
\begin{proof}
(a) If $u = f + v$, it is enough to find $v \in \widetilde{H}^s(\Omega)$ solving the equivalent problem 
\[
B_q(v,w) - \lambda(v,w)_{\mR^n} = \tilde{F}(w), \quad w \in \widetilde{H}^s(\Omega),
\]
for a suitable $\tilde{F} \in (\widetilde{H}^s(\Omega))^*$. Consider the symmetric bilinear form $B_q(v,w)$ for $v, w \in \widetilde{H}^s(\Omega)$. Now $v = I_{s} (-\Delta)^{s/2} v$ for any $v \in H^s$ where $I_s = (-\Delta)^{-s/2}$ is the Riesz potential (since this holds on the dense subset consisting of those $v$ with $\hat{v}Ê= 0$ near $0$), and thus, using the fact that $\Omega$ is bounded and the Hardy-Littlewood-Sobolev inequality,  
\[
\norm{v}_{L^2} \leq C_{\Omega} \norm{v}_{L^{\frac{2n}{n-2s}}} \leq C \norm{(-\Delta)^{s/2} v}_{L^2}, \quad v \in \widetilde{H}^s(\Omega).
\]
If $\mu = \norm{q_-}_{L^{\infty}(\Omega)}$ where $q_-(x) = -\min\{Ê0, q(x) \}$, then for $v \in \widetilde{H}^s(\Omega)$, 
\begin{align*}
B_q(v,v) + \mu(v,v)_{\mR^n} &\geq \norm{(-\Delta)^{s/2} v}_{L^2}^2 \geq c (\norm{v}_{L^2}^2 + \norm{(-\Delta)^{s/2} v}_{L^2}^2 ) \\
 &\geq c \norm{v}_{H^s}^2.
\end{align*}
By the Riesz representation theorem, there is a unique $v = G_{\mu} \tilde{F}$ in $\widetilde{H}^s(\Omega)$ satisfying $B_q(v,w) + \mu(v,w)_{\mR^n}= \tilde{F}(w)$ for $w \in \widetilde{H}^s(\Omega)$. Now 
\[
B_q(v,\,\cdot\,) - \lambda(v,\,\cdot\,) = \tilde{F}(\,\cdot\,) \text{ on $\widetilde{H}^s(\Omega)$} \ \Longleftrightarrow \ v = G_{\mu}\left[ (\mu+\lambda) v + \tilde{F} \right].
\]
The operator $G_{\mu}$ is bounded $(\widetilde{H}^s(\Omega))^* \to \widetilde{H}^s(\Omega)$, and by compact Sobolev embedding it gives rise to a compact, self-adjoint, positive definite operator $L^2(\Omega) \to L^2(\Omega)$. The spectral theorem for compact self-adjoint operators proves (a); in particular the eigenvalues of $G_{\mu}$ are $\{Ê\frac{1}{\lambda_j+\mu} \}_{j=1}^{\infty}$, and $\Sigma \subset (-\norm{q_-}_{L^{\infty}}, \infty)$.

(b) The stated condition is equivalent with 
\[
B_q(u,v) - \lambda(u,v)_{\mR^n} = F(v) \ \text{ for $v \in C^{\infty}_c(\Omega)$}, \qquad u-f \in \widetilde{H}^s(\Omega).
\]
This is equivalent with the condition in (a) since $C^{\infty}_c(\Omega)$ is dense in $\widetilde{H}^s(\Omega)$.

(c) Note that \eqref{dirichlet_uniqueness} states that any solution in $H^s_{\ol{\Omega}}$ is identically zero. This is stronger than stating that any solution in $\widetilde{H}^s(\Omega)$ is zero, which is equivalent with $0 \notin \Sigma$ by the Fredholm alternative. If $q \geq 0$, then the argument in (a), with $\widetilde{H}^s(\Omega)$ replaced by $H^s_{\ol{\Omega}}$, implies \eqref{dirichlet_uniqueness}.
\end{proof}

\subsection*{DN map}

By analogy with the case $s=1$, we may define the DN map for the fractional Schr\"odinger equation via the bilinear form $B_q$ for the equation given in Lemma \ref{lemma_fractional_dirichlet_solvability}.

\begin{Lemma} \label{lemma_dnmap_definition}
Let $\Omega \subset \mR^n$ be a bounded open set, let $0 < s < 1$, and let $q \in L^{\infty}(\Omega)$ satisfy \eqref{dirichlet_uniqueness}. There is a bounded linear map 
\[
\Lambda_q: X \to X^*,
\]
where $X$ is the abstract trace space $X = H^s(\mR^n) / \widetilde{H}^s(\Omega)$, defined by 
\[
(\Lambda_q [f], [g]) = B_q(u_f, g), \qquad f, g \in H^s(\mR^n),
\]
where $u_f \in H^s(\mR^n)$ solves $((-\Delta)^s + q)u = 0$ in $\Omega$ with $u - f \in \widetilde{H}^s(\Omega)$. One has 
\[
(\Lambda_q [f], [g]) = ([f], \Lambda_q [g]), \qquad f, g \in H^s(\mR^n).
\]
\end{Lemma}
\begin{proof}
Let $f, g \in H^s(\mR^n)$. Since $B_q(u_{f+\varphi}, g+\psi) = B_q(u_f, g)$ for $\varphi, \psi$ in $\widetilde{H}^s(\Omega)$, the expression $(\Lambda_q [f], [g]) = B_q(u_f,g)$ is well defined and 
\begin{align*}
\abs{(\Lambda_q [f], [g])} &\leq \norm{(-\Delta)^{s/2} u_f}_{L^2} \norm{(-\Delta)^{s/2} g}_{L^2} + \norm{q}_{L^{\infty}} \norm{u_f}_{L^2} \norm{g}_{L^2} \\
 &\leq C \norm{u_f}_{H^s} \norm{g}_{H^s} \leq C \norm{f}_{H^s} \norm{g}_{H^s}.
\end{align*}
Thus $\abs{(\Lambda_q [f], [g])} \leq CÊ\norm{[f]}_X \norm{[g]}_X$ so $\Lambda_q$ is well-defined and bounded, and self-adjointness follows by taking $g = u_g$.
\end{proof}

If $\Omega$ has Lipschitz boundary, then $X = H^s(\Omega_e)$ and $X^* = H^{-s}_{\ol{\Omega}_e}$ with natural identifications, but functions in $H^{-s}_{\ol{\Omega}_e}$ are only uniquely determined by their restrictions to $\Omega_e$ if $ s < 1/2$. Thus, for Lipschitz domains, one should think of the DN map as an operator 
\[
\Lambda_q: H^s(\Omega_e) \to H^{-s}_{\ol{\Omega}_e}(\mR^n).
\]

The integral identity that allows to solve the inverse problem is a direct consequence of Lemma \ref{lemma_dnmap_definition}. For simplicity, we will write $f$ instead of $[f]$ for elements of $X$.

\begin{Lemma} \label{lemma_integral_identity}
Let $\Omega \subset \mR^n$ be a bounded open set, let $0 < s < 1$, and let $q_1, q_2 \in L^{\infty}(\Omega)$ satisfy \eqref{dirichlet_uniqueness}. For any $f_1, f_2 \in X$ one has 
\[
( (\Lambda_{q_1} - \Lambda_{q_2}) f_1, f_2) = ( (q_1 - q_2) r_{\Omega} u_1, r_{\Omega} u_2)_{\Omega}
\]
where $u_j \in H^s(\mR^n)$ solves $((-\Delta)^s + q_j) u_j = 0$ in $\Omega$ with $u_j|_{\Omega_e} = f_j$.
\end{Lemma}
\begin{proof}
One has 
\begin{align*}
( (\Lambda_{q_1} - \Lambda_{q_2}) f_1, f_2) &= ( \Lambda_{q_1} f_1, f_2) - ( f_1, \Lambda_{q_2} f_2) 
 = B_{q_1}(u_1,u_2) - B_{q_2}(u_1,u_2) \\
  &= ( (q_1 - q_2) r_{\Omega} u_1, r_{\Omega} u_2)_{\Omega}. \qedhere
\end{align*}
\end{proof}

\section{The DN map} \label{sec_dnmap_description}

The abstract definition of the DN map $\Lambda_q$ in Section \ref{sec_preliminaries} is sufficient for the formulation and solution of the inverse problem. However, in this section we will give more concrete descriptions of the DN map, valid under stronger regularity assumptions. For simplicity we assume that the boundary and the potential are $C^{\infty}$.

\subsection*{DN map and $(-\Delta)^s$}

\begin{Lemma} \label{lemma_dnmap_pointwise}
Let $\Omega \subset \mR^n$ be a bounded open set with $C^{\infty}$ boundary, let $0 < s < 1$, and let $q \in C^{\infty}_c(\Omega)$ satisfy \eqref{dirichlet_uniqueness}. For any $\beta \geq 0$ satisfying $s-1/2 < \beta < 1/2$, the restriction of $\Lambda_q$ to $H^{s+\beta}(\Omega_e)$ is the map 
\[
\Lambda_q: H^{s+\beta}(\Omega_e) \to H^{-s+\beta}(\Omega_e), \ \ \Lambda_q f = (-\Delta)^s u_f|_{\Omega_e}
\]
where $u_f \in H^{s+\beta}(\mR^n)$ solves $((-\Delta)^s + q)u = 0$ in $\Omega$ with $u|_{\Omega_e} = f$.
\end{Lemma}
\begin{proof}
First we use a result from \cite{VishikEskin}, see also \cite{Grubb}: if $\beta \in [0, 1/2)$, then for any $f \in H^{s+\beta}(\Omega_e)$ there is a unique $u = u_f \in H^{s+\beta}(\mR^n)$ satisfying 
\[
((-\Delta)^s + q) u = 0 \text{ in $\Omega$}, \qquad u|_{\Omega_e} = f.
\]
In fact \cite[Theorem 3.1]{Grubb} asserts Fredholm solvability for the inhomogeneous problem, but the result above can be reduced to this case by taking a $H^{s+\beta}$ extension of $f$ to $\mR^n$, and Fredholm solvability implies unique solvability since the finite dimensional kernel and range complement are independent of $\beta$ by \cite[Theorem 3.5]{Grubb2} and they are trivial when $\beta = 0$ by Lemma \ref{lemma_fractional_dirichlet_solvability}.

Now for $f, g \in H^{s+\beta}(\Omega_e)$ with $\beta \in [0,1/2)$, let $u_f \in H^{s+\beta}(\mR^n)$ be the solution obtained above and let $e_g \in H^{s+\beta}(\mR^n)$ be some extension of $g$. Then, by definition, 
\begin{align*}
(\Lambda_q f, g) &= ((-\Delta)^{s/2} u_f, (-\Delta)^{s/2} e_g)_{\mR^n} + (q r_{\Omega} u_f, r_{\Omega} e_g)_{\Omega} \\
 &= ( (-\Delta)^s u_f, e_g )_{\mR^n} + (q r_{\Omega} u_f, r_{\Omega} e_g)_{\Omega}
\end{align*}
since $((-\Delta)^{s/2} u, (-\Delta)^{s/2} v)_{\mR^n} = ((-\Delta)^s u, v)_{\mR^n}$ holds first for Schwartz functions by the Parseval identity, and then also for $u, v \in H^s(\mR^n)$ by density.

It remains to show that whenever $\alpha \in (-1/2, 1/2)$, $u \in H^{-\alpha}(\mR^n)$, $v \in H^{\alpha}(\mR^n)$, then 
\begin{equation} \label{halpha_omega_splitting}
(u,v)_{\mR^n} = (r_{\Omega} u, r_{\Omega} v)_{\Omega} + (r_{\Omega_e} u, r_{\Omega_e} v)_{\Omega_e}
\end{equation}
in the sense of distributional pairings. If \eqref{halpha_omega_splitting} is true, then the assumption $\beta \in (s-1/2,1/2)$ implies $(-\Delta)^s u_f \in H^{-s+\beta}(\mR^n)$ with $-s+\beta \in (-1/2,1/2)$, and since $u_f$ is a solution in $\Omega$ one has 
\[
(\Lambda_q f, g) = ( (-\Delta)^s u_f, e_g )_{\mR^n} + (q r_{\Omega} u_f, r_{\Omega} e_g)_{\Omega} = (r_{\Omega_e} (-\Delta)^s u_f, g)_{\Omega_e}
\]
which concludes the proof.

To show \eqref{halpha_omega_splitting}, let $\chi_{\Omega}$ be the characteristic function of $\Omega$. This is a pointwise multiplier on $H^{\gamma}(\mR^n)$ for $\gamma \in (-1/2,1/2)$ \cite{Triebel_Lipschitz}, and the same is true for $1-\chi_{\Omega}$. We may write $u = \chi_{\Omega} u + (1-\chi_{\Omega}) u$ and similarly for $v$, and then 
\[
(u,v)_{\mR^n} = (\chi_{\Omega} u, \chi_{\Omega} v)_{\mR^n} + ((1-\chi_{\Omega}) u, (1-\chi_{\Omega}) v)_{\mR^n} 
\]
where the cross terms vanish first for Schwartz $u$, $v$ and then in general by density. Now $\chi_{\Omega} u$ is in $H^{-\alpha}_{\ol{\Omega}}$, hence can be approximated by functions in $C^{\infty}_c(\Omega)$. Using similar approximations for the other functions and restricting to $\Omega$ and $\Omega_e$ implies \eqref{halpha_omega_splitting}.
\end{proof}

\subsection*{DN map and $\mathcal{N}_s$}
Several nonlocal Neumann boundary operators appear in the literature, see \cite{DipierroRosOtonValdinoci, Grubb_spectral} and references therein. We will relate $\Lambda_q$ to the nonlocal Neumann boundary operator $\mathcal{N}_s$ introduced in \cite{DipierroRosOtonValdinoci}, defined pointwise by 
\begin{equation} \label{nonlocal_neumann_pointwise}
\mathcal{N}_s u(x) = c_{n,s} \int_{\Omega} \frac{u(x)-u(y)}{\abs{x-y}^{n+2s}} \,dy, \qquad x \in \Omega_e.
\end{equation}
The next lemma contains a definition that applies to Sobolev functions. The result states that knowing $\Lambda_q f|_W$ for $f \in C^{\infty}_c(W)$ is equivalent to knowing $\mathcal{N}_s u_f|_W$ for $f \in C^{\infty}_c(W)$, since $\Lambda_q f|_W$ and $\mathcal{N}_s u_f|_W$ only differ by quantities that do not depend on the unknown potential $q$.

\begin{Lemma}
Assume the conditions in Lemma \ref{lemma_dnmap_pointwise}. One has 
\[
\Lambda_q f = \mathcal{N}_s u_f - mf + (-\Delta)^s(E_0 f)|_{\Omega_e}, \qquad f \in H^{s+\beta}(\Omega_e)
\]
where, for $\gamma > -1/2$, $\mathcal{N}_s$ is the map 
\[
\mathcal{N}_s: H^{\gamma}(\mR^n) \to H^{\gamma}_{\mathrm{loc}}(\Omega_e), \ \ \mathcal{N}_s u = m u|_{\Omega_e} + (-\Delta)^s(\chi_{\Omega} u)|_{\Omega_e}
\]
where $m \in C^{\infty}(\Omega_e)$ is given by $m(x) = c_{n,s} \int_{\Omega} \frac{1}{\abs{x-y}^{n+2s}} \,dy$ and $\chi_{\Omega}$ is the characteristic function of $\Omega$. Also, $E_0$ is extension by zero. If $u \in L^2(\mR^n)$, then $\mathcal{N}_s u \in L^2_{\mathrm{loc}}(\Omega_e)$ is given a.e.\ by the formula \eqref{nonlocal_neumann_pointwise}.
\end{Lemma}
\begin{proof}
If $u \in H^{\gamma}(\mR^n)$ with $\gamma > -1/2$, then $mu|_{\Omega_e} \in H^{\gamma}_{\mathrm{loc}}(\Omega_e)$. By the pointwise multiplier property of $\chi_{\Omega}$, we have $\chi_{\Omega} u \in H^{\alpha}(\mR^n)$ for some $\alpha \in (-1/2,1/2)$ and $(-\Delta)^s (\chi_{\Omega} u) \in H^{\alpha-2s}(\mR^n)$. However, if $\varphi, \psi \in C^{\infty}_c(\mR^n)$ satisfy $\varphi = 1$ near $\ol{\Omega}$ and $\psi = 1$ near $\mathrm{supp}(\varphi)$, then for any $r, t \in \mR$ one has 
\[
(1-\psi)(-\Delta)^s \varphi: H^{-r}(\mR^n) \to H^t(\mR^n)
\]
by the pseudolocal property of Fourier multipliers. Thus one also has $(-\Delta)^s (\chi_{\Omega} u)|_{\Omega_e} \in H^t_{\mathrm{loc}}(\Omega_e)$ for any $t$, and $\mathcal{N}_s$ is well-defined and maps $H^{\gamma}(\mR^n)$ to $H^{\gamma}_{\mathrm{loc}}(\Omega_e)$ for $\gamma > -1/2$.

Moreover, if $u \in L^2(\mR^n)$ and if $\varphi_j \in C^{\infty}_c(\Omega)$ satisfy $\varphi_j \to \chi_{\Omega} u$ in $L^2(\mR^n)$, then the pseudolocal property implies that 
\[
(-\Delta)^s(\varphi_j)|_{\Omega_e} \to (-\Delta)^s(\chi_{\Omega} u)|_{\Omega_e} \quad \text{in $L^2_{\mathrm{loc}}(\Omega_e)$.}
\]
After extracting a subsequence (using the diagonal argument), one has convergence a.e.\ in $\Omega_e$. Thus the pointwise expression \eqref{nonlocal_neumann_pointwise} for a.e.\ $x \in \Omega_e$ follows from the standard formula 
\[
(-\Delta)^s \varphi(x) = c_{n,s} \int_{\mR^n} \frac{\varphi(x)-\varphi(y)}{\abs{x-y}^{n+2s}} \,dy, \quad \varphi \in C^{\infty}_c(\Omega), \ \ x \in \Omega_e.
\]

Let us prove the formula for $\Lambda_q$. If $f \in H^{s+\beta}(\Omega_e)$, then $f \in H^{\alpha}(\Omega_e)$ for some $\alpha \in (-1/2,1/2)$ and hence $E_0 f, u_f \in H^{\alpha}(\mR^n)$. Recall also that $\chi_{\Omega}$ and $1-\chi_{\Omega}$ are pointwise multipliers on $H^{\alpha}(\mR^n)$. Then 
\begin{align*}
\Lambda_q f &= (-\Delta)^s u_f|_{\Omega_e} = (-\Delta)^s (\chi_{\Omega} u_f)|_{\Omega_e} + (-\Delta)^s ((1-\chi_{\Omega})u_f)|_{\Omega_e} \\
 &= \mathcal{N}_s u_f - m f + (-\Delta)^s (E_0 f)|_{\Omega_e}. \qedhere
\end{align*}
\end{proof}

\subsection*{Nonlocal diffusion}

Finally, we will give a heuristic interpretation of the quantity $\Lambda_q f(x)$ in terms of nonlocal diffusions \cite{AMRT}. This discussion is mostly for illustrative purposes, so we will not give precise arguments and will restrict to the case $q = 0$.

We begin with a macroscopic description of nonlocal diffusion in $\mR^n$. Suppose that $u(x,t)$ describes the density of particles at a point $x \in \mR^n$ at time $t$. Given an initial density $u_0(x)$, we assume that $u(x,t)$ is obtained as a solution of the nonlocal diffusion equation 
\begin{equation} \label{diffusion_rn}
\left\{ \begin{array}{rl}
\partial_t u + (-\Delta)^s u &\!\!\!= 0 \text{ in $\mR^n \times \{ t > 0 \}$,} \\
u|_{t=0} &\!\!\!= u_0.
\end{array} \right.
\end{equation}
Taking Fourier transforms in $x$, the solution at time $t$ is given by 
\[
u(t,x) = (p_t \ast u_0)(x)
\]
where $p_t(x) = \mF^{-1} \{ e^{-t \abs{\xi}^{2s}} \}$ is the probability density function of the L\'evy process $X_t$ with infinitesimal generator $-(-\Delta)^s$. If $s=1$, $p_t$ is a Gaussian, but for $0 < s < 1$ it is a heavy-tailed distribution with $p_t(x) \sim \abs{x}^{-n-2s}$ for large $\abs{x}$ (for $s=1/2$, $p_t(x) = c_n t(t^2+\abs{x}^2)^{-\frac{n+1}{2}}$). The L\'evy process $X_t$ also gives a microscopic description of $u(x,t)$: it is obtained as the expected value 
\[
u(x,t) = \mathbb{E}_x \left[Êu_0(X_t) \right]
\]
which expresses how many L\'evy particles from the initial distribution $u_0$ have jumped to $x$ at time $t$. See \cite{Applebaum, Chen2010} for L\'evy processes.

Let now $\Omega \subset \mR^n$ be a bounded open set. We consider the following Dirichlet problem for nonlocal diffusion: given $u_0 \in H^s_{\ol{\Omega}}$, find $u$ so that 
\begin{equation} \label{diffusion_omega_homogeneous}
\left\{ \begin{array}{rl}
\partial_t u + (-\Delta)^s u &\!\!\!= 0 \text{ in $\Omega \times \{ t > 0 \}$}, \\
u|_{\Omega_e \times \{Êt > 0 \}} &\!\!\!= 0, \\
u|_{\mR^n \times \{t=0\}} &\!\!\!= u_0.
\end{array} \right.
\end{equation}
The solution is easily obtained in the form 
\begin{equation} \label{diffusion_omega_homogeneous_solution}
u(x,t) = \sum_{j=1}^{\infty} e^{-\lambda_j t} c_j \phi_j(x)
\end{equation}
where $u_0 = \sum_{j=1}^{\infty} c_j \phi_j$ and $\{Ê\phi_j \}_{j=1}^{\infty} \subset H^s_{\ol{\Omega}}$ is an orthonormal basis of $L^2_{\ol{\Omega}}$ consisting of eigenfunctions for $(-\Delta)^s$ with eigenvalues $\lambda_j$, so that $(-\Delta)^s \phi_j = \lambda_j \phi_j$ in $\Omega$, $\phi_j|_{\Omega_e} = 0$, and $0 < \lambda_1 \leq \lambda_2 \leq \cdots \to \infty$. The probabilistic interpretation is that we are looking at L\'evy particles in $\Omega$ that are terminated when they reach the exterior. One has 
\[
u(x,t) = \mathbb{E}_x \left[Êu_0(X_t) 1_{\{t < \tau\}} \right]
\]
where $\tau$ is the time when the L\'evy process exits $\Omega$.

By the Duhamel principle and a standard reduction to homogeneous Dirichlet values, given any $f \in H^s(\Omega_e)$ and any $e_f \in H^s(\mR^n)$ with $e_f|_{\Omega_e} = f$, we can also solve the equation 
\begin{equation} \label{diffusion_omega_fixedf}
\left\{ \begin{array}{rll}
\partial_t v + (-\Delta)^s v &\!\!\!= 0 &\text{ in $\Omega \times \{ t > 0 \}$,} \\
v(\,\cdot\,,t)|_{\Omega_e} &\!\!\!= f &\text{ for $t > 0$,}
\end{array} \right.
\end{equation}
with initial value $v|_{\mR^n \times \{t=0\}} = e_f$. Another solution of \eqref{diffusion_omega_fixedf} is given by $v_s(x,t) = u_f(x)$, if $u_f \in H^s(\mR^n)$ solves $(-\Delta)^s u = 0$ in $\Omega$ with $u|_{\Omega_e} = f$. The function $u_f$ is the unique steady state of \eqref{diffusion_omega_fixedf}, since $v-v_s$ solves \eqref{diffusion_omega_homogeneous} for some $u_0$, and \eqref{diffusion_omega_homogeneous_solution} implies that 
\[
\norm{v(\,\cdot\,,t) - u_f}_{H^s} \to 0 \text{ as $t \to \infty$}.
\]

Now, given $f \in H^s(\Omega_e)$ and the solution $u_f$ of the Dirichlet problem, we may consider two nonlocal diffusions with initial value $u_f$:
\begin{itemize}
\item 
the free diffusion \eqref{diffusion_rn} in $\mR^n$ with solution $u(x,t)$,
\item 
the diffusion \eqref{diffusion_omega_fixedf} whose exterior value is fixed to be $f$.
\end{itemize}
If $t$ is small and $x \in \Omega_e$, $u(x,t)$ formally satisfies 
\begin{align*}
u(x,t) &= u(x,0) + \partial_t u(x,0) t + O(t^2) = f(x) - (-\Delta)^s u(x,0) t + O(t^2) \\
 &= f(x) - (\Lambda_0 f)(x) t + O(t^2)
\end{align*}
by Lemma \ref{lemma_dnmap_pointwise}. Thus the DN map may be interpreted as follows:
\begin{itemize}
\item 
$-\Lambda_0 f(x)$ is the (infinitesimal) amount of particles migrating to $x$ in the free diffusion that starts from the steady state $u_f$;
\item 
$\Lambda_0 f(x)$ is the (infinitesimal) cost required to maintain the exterior value $f$ at $x$ in the steady state nonlocal diffusion.
\end{itemize}
Similar remarks apply to $\Lambda_q$ at least if $q \geq 0$. We refer to \cite{Chen} for some facts on the related stochastic processes, and to \cite{PiiroinenSimon} for stochastic interpretations of the usual Calder\'on problem.

\section{Uniqueness properties} \label{sec_uniqueness}

We prove the uniqueness result for the fractional Laplacian, Theorem \ref{thm_uniqueness}, which is an easy consequence of the Carleman estimates in \cite{Ruland} and the Caffarelli-Silvestre extension \cite{CaffarelliSilvestre}.

\begin{proof}[Proof of Theorem \ref{thm_uniqueness}]
Assume first that $u$ is a continuous bounded function in $\mR^n$. Write $\mR^{n+1}_+ = \{Ê(x,y) \,;\, x \in \mR^n, \ y > 0 \}$, and denote by $w$ the extension of $u$ to $\mR^{n+1}_+$ defined by 
\[
w(x,y) = (P_y \ast u)(x), \qquad P_y(x) = c_{n,s}\frac{y^{2s}}{(\abs{x}^2 + y^2)^{\frac{n+2s}{2}}}.
\]
By \cite[Remark 3.8]{CabreSire1} $w$ is the unique continuous bounded solution in $\overline{\mR^{n+1}}$ of the Dirichlet problem 
\[
\mathrm{div}(y^{1-2s} \nabla w) = 0 \text{ in $\mR^{n+1}$}, \qquad w|_{y=0} = u.
\]
If we additionally assume that $u \in H^s(\mR^n)$, then by \cite[Section 3]{CabreSire1} the solution $w$ satisfies $\int_{\mR^{n+1}_+} y^{1-2s} \abs{\nabla w}^2 \,dx \,dy < \infty$, and one has 
\[
(-\Delta)^s u = - d_s \lim_{y \to 0^+} y^{1-2s} \partial_y w(\,\cdot\,,y)
\]
where the limit exists in $H^{-s}(\mR^n)$. See \cite{CabreSire1} for the precise values of the constants $c_{n,s}$ and $d_s$.

Assume now that $u$ is a continuous bounded function in $\mR^n$ with $u \in H^s(\mR^n)$, and $u|_W = (-\Delta)^s u|_W = 0$ where $W$ is a ball in $\mR^n$. Denote by $B$ the ball in $\mR^{n+1}$ with $B \cap \{Êy=0 \} = W$, and define $B^+ = \{Ê(x,y) \in B \,;\, y > 0 \}$. Since $u|_W = (-\Delta)^s u|_W = 0$, $w$ satisfies 
\[
\mathrm{div}(y^{1-2s} \nabla w) = 0 \text{ in $B^+$}, \ \ w|_{B \cap \{ y=0 \}} = \lim_{y \to 0^+} y^{1-2s} \partial_y w|_{B \cap \{ y=0 \}} = 0.
\]
The function $w$ thus satisfies the conditions in \cite[Proposition 2.2]{Ruland}, and one obtains that $w|_{B^+} \equiv 0$. But $w$ is real-analytic in $\mR^{n+1}_+$ as the solution of an elliptic equation with real-analytic coefficients (see \cite[Theorem 8.6.1]{Hormander1}). Hence $w \equiv 0$ in $\mR^{n+1}$, which implies that $u \equiv 0$.

Finally, let $u \in H^{-r}(\mR^n)$ for some $r > 0$, and $u|_W = (-\Delta)^s u|_W = 0$ for some ball $W \subset \mR^n$. Consider the smooth approximations  
\[
u_{\eps} = u \ast \eps^{-n} \varphi(\,\cdot\,/\eps)
\]
where $\varphi \in C^{\infty}_c(\mR^n)$ satisfies $\int \varphi \,dx = 1$, $\varphi \geq 0$, and $\varphi = 0$ for $\abs{x} \geq 1$. There exist $\eps_0 > 0$ and a smaller ball $W' \subset W$ such that $u_{\eps}|_{W'} = 0$ and also $(-\Delta)^s u_{\eps}|_{W'} = ((-\Delta)^s u) \ast \eps^{-n} \varphi(\,\cdot\,/\eps)|_{W'} = 0$ whenever $\eps < \eps_0$. Now each $u_{\eps}$ is in $H^{\alpha}(\mR^n)$ for any $\alpha \in \mR$, since $\hat{u}_{\eps}(\xi) = m(\xi) \hat{u}(\xi)$ where $m(\xi) = \hat{\varphi}(\eps \xi)$ is a Schwartz function and $\br{\xi}^{-r} \hat{u}(\xi)$ is in $L^2$. By Sobolev embedding, each $u_{\eps}$ is also continuous and bounded in $\mR^n$. The argument above implies that $u_{\eps} \equiv 0$ whenever $\eps < \eps_0$, showing that $u = \lim_{\eps \to 0} u_{\eps} = 0$.
\end{proof}

\begin{Remark}
We note that for $s=1/2$ the above argument simplifies: the function $w$ in the proof is just the harmonic extension of $u$ to $\mR^{n+1}$, and it satisfies $w|_{W \times \{Êy=0 \}} = \partial_y w|_{W \times \{Êy=0 \}} = 0$. The odd extension $\tilde{w}$ of $w$ to $W \times \mR$ is smooth, satisfies $\Delta_{x,y} \tilde{w} = 0$, and $\tilde{w}|_{W \times \{Êy=0 \}} = \partial_y \tilde{w}|_{W \times \{Êy=0 \}} = 0$. Using the equation one observes that $\tilde{w}$ vanishes to infinite order on $W \times \{Êy=0 \}$, thus by analyticity $\tilde{w} \equiv 0$ and $u \equiv 0$.
\end{Remark}

\begin{Remark}
For comparison, we recall the original argument in \cite[Chapitre III.11]{Riesz} for proving a result like Theorem \ref{thm_uniqueness}. There are two steps: first one uses the Kelvin transform to reduce to the case where $u$ and $(-\Delta)^s u$ vanish outside some ball, and then one computes derivatives of $u$ and lets $x \to \infty$ to show that all moments of $(-\Delta)^s u$ must vanish. See \cite[Lemma 3.5.4]{Isakov_source} for another proof of the second step.

Let $u$ be in the Sobolev space $W^{-r,q}(\mR^n)$ for some $r \in \mR$, where $q = \frac{2n}{n+2s}$. By approximation, translation and scaling, we may assume that $u \in W^{t,q}(\mR^n)$ for any $t > 0$ and $u|_B = (-\Delta)^s u|_B = 0$ where $B$ is the unit ball. Write $f = (-\Delta)^s u$, so $f, u \in L^q \cap L^{\infty}$ and $u = I_{2s} f$. Define 
\[
v = R_{2s} u, \qquad g = R_{-2s} f
\]
where $R_{\alpha} f(x) = \abs{x}^{\alpha-n} f(K(x))$ and $K(x) = x/\abs{x}^2$ is the Kelvin transform. Since $\det DK(x) = -\abs{x}^{-2n}$ and $\abs{K(x)-K(y)} = \frac{\abs{x-y}}{\abs{x} \abs{y}}$, one computes $\norm{R_{-2s} f}_{L^q} = \norm{f}_{L^q}$ and $R_{2s} I_{2s} f = I_{2s} R_{-2s} f$. Then $g \in L^q$, both $v = I_{2s} g$ and $g$ vanish outside $B$, and 
\[
v(x) = c_{n,s} \int_{B} \abs{x-y}^{2s-n} g(y) \,dy = 0, \qquad \abs{x} > 1.
\]
In particular, letting $x \to \infty$, one gets $\int_B g(y) \,dy = 0$. Applying powers of the Laplacian to $v(x)$ we get 
\[
\int_B \abs{x-y}^{2s-n-2k} g(y) \,dy = 0, \qquad k \geq 0, \ \ \abs{x} > 1.
\]
Computing $\partial_{x_j} v(x)$ and letting $x \to \infty$ gives $\int_B y_j g(y) \,dy = 0$. Repeating this for higher order derivatives implies that $\int_B y^{\alpha} g(y) \,dy = 0$ for any multi-index $\alpha$, hence $g \equiv 0$. This finally gives $f \equiv 0$ and $u \equiv 0$.

The above argument seems to require that $f \in L^q$ for $q$ close to $1$ in order for $R_{-2s} f$ to be an $L^p$ function for some $p$. If one starts with a solution $u \in H^{-r}$ for some $r$, after approximation one gets $f \in L^2 \cap L^{\infty}$ and there is an issue since $R_{-2s} f$ might have a non-integrable singularity at $0$.
\end{Remark}

\section{Approximation in $L^2(\Omega)$} \label{sec_approximation_ltwo}

We will use the following Runge approximation property for solutions of the fractional Schr\"odinger equation. If $q \in L^{\infty}(\Omega)$ satisfies \eqref{dirichlet_uniqueness}, we denote by $P_q$ the Poisson operator 
\begin{equation} \label{poisson_operator_definition}
P_q: X \to H^s(\mR^n), \ \ f \mapsto u
\end{equation}
where $X = H^s(\mR^n)/\widetilde{H}^s(\Omega)$ is the abstract space of exterior values, and $u \in H^s(\mR^n)$ is the unique solution of $((-\Delta)^s + q)u = 0$ in $\Omega$ with $u - f \in \widetilde{H}^s(\Omega)$ given in Lemma \ref{lemma_fractional_dirichlet_solvability}.

\begin{Lemma} \label{lemma_runge_fractional}
Let $\Omega \subset \mR^n$ be bounded open set, let $0 < s < 1$, and let $q \in L^{\infty}(\Omega)$ satisfy \eqref{dirichlet_uniqueness}. Let also $W$ be any open subset of $\Omega_e$. Consider the set 
\begin{align*}
\mathcal{R} = \{ u|_{\Omega} \,;\, u = P_q f, \ f \in C^{\infty}_c(W) \}.
\end{align*}
Then $\mathcal{R}$ is dense in $L^2(\Omega)$.
\end{Lemma}
\begin{proof}
By the Hahn-Banach theorem, it is enough to show that any $v \in L^2(\Omega)$ with $(v,w)_{\Omega} = 0$ for all $w \in \mathcal{R}$ must satisfy $v \equiv 0$. If $v$ is such a function, then 
\begin{equation} \label{runge_first_fact}
(v, r_{\Omega} P_q f)_{\Omega} = 0, \qquad f \in C^{\infty}_c(W).
\end{equation}

We claim that the formal adjoint of $r_{\Omega} P_q$ is given by 
\begin{equation} \label{runge_second_fact}
(v, r_{\Omega} P_q f)_{\Omega} = -B_q(\varphi, f), \qquad f \in C^{\infty}_c(W),
\end{equation}
where $\varphi \in H^s(\mR^n)$ is the solution given by Lemma \ref{lemma_fractional_dirichlet_solvability} of 
\[
\text{$((-\Delta)^s+q)\varphi = v$ in $\Omega$, \qquad $\varphi \in \widetilde{H}^s(\Omega)$.}
\]
In other words, $B_q(\varphi, w) = (v,r_{\Omega} w)_{\Omega}$ for any $w \in \widetilde{H}^s(\Omega)$. To prove \eqref{runge_second_fact}, let $f \in C^{\infty}_c(W)$, and let $u_f = P_q f \in H^s(\mR^n)$ so $u_f-f \in \widetilde{H}^s(\Omega)$. Then 
\[
(v, r_{\Omega} P_q f)_{\Omega} = (v, r_{\Omega} (u_f - f))_{\Omega} = B_q(\varphi, u_f - f) = -B_q(\varphi,f).
\]
In the last line, we used that $u_f$ is a solution and $\varphi \in \widetilde{H}^s(\Omega)$.

Combining \eqref{runge_first_fact} and \eqref{runge_second_fact}, we have that 
\[
B_q(\varphi, f) = 0, \qquad f \in C^{\infty}_c(W).
\]
Since $r_{\Omega} f = 0$, this implies that 
\[
0 = ((-\Delta)^{s/2} \varphi, (-\Delta)^{s/2} f)_{\mR^n} = ( (-\Delta)^s \varphi, f )_{\mR^n}, \qquad f \in C^{\infty}_c(W).
\]
In particular, $\varphi \in H^s(\mR^n)$ satisfies 
\[
\varphi|_W = (-\Delta)^s \varphi|_W = 0.
\]
Theorem \ref{thm_uniqueness} implies that $\varphi \equiv 0$, and thus also $v \equiv 0$.
\end{proof}

\section{Inverse problem} \label{sec_inverse_problem}

It is now easy to prove the uniqueness result for the inverse problem.

\begin{proof}[Proof of Theorem \ref{thm_main}]
Note that if $F \in X^*$, then $F|_{W_2}$ is a distribution in $W_2$ with $F|_{W_2}(\varphi) = F([\varphi])$, $\varphi \in C^{\infty}_c(W_2)$. Now if $\Lambda_{q_1} f|_{W_2} = \Lambda_{q_2} f|_{W_2}$ for any $f \in C^{\infty}_c(W_1)$, the integral identity in Lemma \ref{lemma_integral_identity} yields that 
\[
\int_{\Omega} (q_1 - q_2) u_1 u_2 \,dx = 0
\]
whenever $u_j \in H^s(\mR^n)$ solve $( (-\Delta)^s + q_j ) u_j = 0$ in $\Omega$ with exterior values in $C^{\infty}_c(W_j)$. Let $f \in L^2(\Omega)$, and use the approximation result, Lemma \ref{lemma_runge_fractional}, to find sequences $(u_j^{(k)})$ of functions in $H^s(\mR^n)$ that satisfy 
\begin{gather*}
( (-\Delta)^s + q_1 ) u_1^{(k)} = ( (-\Delta)^s + q_2) u_2^{(k)} = 0 \text{ in $\Omega$}, \\
\text{$u_j^{(k)}$ have exterior values in $C^{\infty}_c(W_j)$,} \\
r_{\Omega} u_1^{(k)} = f + r_1^{(k)}, \ \ r_{\Omega} u_2^{(k)} = 1 + r_2^{(k)}
\end{gather*}
where $r_1^{(k)}, r_2^{(k)}Ê\to 0$ in $L^2(\Omega)$ as $k \to \infty$. Inserting these solutions in the integral identity and taking the limit as $k \to \infty$ implies that 
\[
\int_{\Omega} (q_1 - q_2) f \,dx = 0.
\]
Since $f \in L^2(\Omega)$ was arbitrary, we conclude that $q_1 = q_2$.
\end{proof}

\section{Higher order approximation} \label{sec_approximation_sobolev}

We proceed to prove Theorem \ref{thm_approximation}(b). The argument is similar to that in Section \ref{sec_approximation_ltwo}, but since the approximation is in high regularity spaces, by duality we will need to solve Dirichlet problems with data in negative order Sobolev spaces. This follows again by duality from regularity results for the Dirichlet problem proved in \cite{Hormander_unpublished, Grubb}.

We will next introduce function spaces from \cite{Grubb}. To keep closer to the notation of \cite{Grubb}, in this section we write the fractional Laplacian as $(-\Delta)^a$ where $0 < a < 1$. Assume that $\Omega \subset \mR^n$ is a bounded domain with $C^{\infty}$ boundary, and let $q \in C^{\infty}_c(\Omega)$ satisfy the analogue of \eqref{dirichlet_uniqueness}, 
\begin{equation} \label{dirichlet_uniqueness_a}
\left\{Ê\begin{array}{c} \text{if $u \in H^a(\mR^n)$ solves $((-\Delta)^a + q)u = 0$ in $\Omega$ and $u|_{\Omega_e} = 0$,} \\
\text{then $u \equiv 0$.} \end{array} \right.
\end{equation}
We assume $q$ compactly supported to fit the operator theory in \cite{Grubb}. Define 
\[
\mathcal{E}_a(\ol{\Omega}) = e^+ d(x)^a C^{\infty}(\ol{\Omega})
\]
where $e^+$ denotes extension by zero from $\Omega$ to $\mR^n$, and $d$ is a $C^{\infty}$ function in $\ol{\Omega}$, positive in $\Omega$â and satisfying $d(x) = \mathrm{dist}(x,\partial \Omega)$ near $\partial \Omega$. If $s > a-1/2$ we will also consider the Banach space $H^{a(s)}(\ol{\Omega})$ which arises as the exact solution space of functions $u$ satisfying 
\[
r_{\Omega}((-\Delta)^a + q) u \in H^{s-2a}(\Omega), \qquad u|_{\Omega_e} = 0.
\]
We will not give the actual definition, but instead we will use the following properties from \cite{Grubb}.

\begin{Lemma} \label{lemma_has_spaces}
For any $s > a-1/2$, there is a Banach space $H^{a(s)}(\ol{\Omega})$ with the following properties:
\begin{enumerate}
\item[(a)] 
$H^{a(s)}(\ol{\Omega}) \subset H^{a-1/2}_{\ol{\Omega}}$ with continuous inclusion,
\item[(b)]  
$H^{a(s)}(\ol{\Omega}) = H^s_{\ol{\Omega}}$ if $s \in (a-1/2, a+1/2)$,
\item[(c)]  
the operator $r_{\Omega} ((-\Delta)^a + q)$ is a homeomorphism from $H^{a(s)}(\ol{\Omega})$ onto $H^{s-2a}(\Omega)$,
\item[(d)]  
$H^s_{\ol{\Omega}} \subset H^{a(s)}(\ol{\Omega}) \subset H^s_{\mathrm{loc}}(\Omega)$ with continuous inclusions, i.e.\ multiplication by any $\chi \in C^{\infty}_c(\Omega)$ is bounded $H^{a(s)}(\ol{\Omega}) \to H^s(\Omega)$,
\item[(e)]
$\mathcal{E}_a(\ol{\Omega}) = \cap_{s > a-1/2} H^{a(s)}(\ol{\Omega})$, and $\mathcal{E}_a(\ol{\Omega})$ is dense in $H^{a(s)}(\ol{\Omega})$.
\end{enumerate}
\end{Lemma}
\begin{proof}
(a) and (b) follow from \cite[Section 1]{Grubb}. (c) follows since $r_{\Omega} ((-\Delta)^a + q): H^{a(s)}(\ol{\Omega}) \to H^{s-2a}(\Omega)$ is a Fredholm operator \cite[Theorem 2]{Grubb}, it has a finite-dimensional kernel and range complement  independent of $s$ \cite[Theorem 3.5]{Grubb2}, and for $s=a$ the kernel and range complement are trivial using \eqref{dirichlet_uniqueness_a} and Lemma \ref{lemma_fractional_dirichlet_solvability}. (d) follows from (c) and (a), or alternatively from the definitions in \cite[Section 1]{Grubb}. (e) is in \cite[Proposition 4.1]{Grubb}.
\end{proof}

We next prove an approximation result in the space $\mathcal{E}_a(\ol{\Omega})$, equipped with the topology induced by the norms $\{ \norm{\,\cdot\,}_{H^{a(m)}(\ol{\Omega})} \}_{m=1}^{\infty}$. Then $\mathcal{E}_a(\ol{\Omega})$ is a Fr\'echet space.

\begin{Lemma} \label{lemma_runge_fractional_higherorder}
Let $\Omega \subset \mR^n$, $n \geq 2$, be a bounded domain with $C^{\infty}$ boundary, let $0 < a < 1$, let $W$ be an open subset of $\Omega_e$, and let $q \in C^{\infty}_c(\Omega)$ satisfy \eqref{dirichlet_uniqueness_a}. If $P_q$ is the Poisson operator in \eqref{poisson_operator_definition}, define 
\[
\mathcal{R} = \{ e^+ r_{\Omega} P_q f \,;\, f \in C^{\infty}_c(W) \}.
\]
Then $\mathcal{R}$ is a dense subset of $\mathcal{E}_a(\ol{\Omega})$.
\end{Lemma}
\begin{proof}
Note that $\mathcal{R} \subset \mathcal{E}_a(\ol{\Omega})$, since for $f \in C^{\infty}_c(W)$ one has $P_q f = f + v$ where $r_{\Omega}((-\Delta)^a + q)v \in C^{\infty}(\ol{\Omega})$ and $v|_{\Omega_e} = 0$, hence $v \in \mathcal{E}_a(\ol{\Omega})$ by Lemma \ref{lemma_has_spaces}.

Let $L$ be a continuous linear functional on $\mathcal{E}_a(\ol{\Omega})$ that satisfies 
\[
L(e^+ r_{\Omega} P_q f) = 0, \qquad f \in C^{\infty}_c(W).
\]
It is enough to show that $L \equiv 0$, since then $\mathcal{R}$ will be dense by the Hahn-Banach theorem.

By the properties of Fr\'echet spaces, there exists an integer $s$ so that 
\[
\abs{L(u)} \leq C \sum_{m=1}^s \norm{u}_{H^{a(m)}(\ol{\Omega})} \leq C' \norm{u}_{H^{a(s)}(\ol{\Omega})}, \qquad u \in \mathcal{E}_a(\ol{\Omega}).
\]
Since $\mathcal{E}_a(\ol{\Omega})$ is dense in $H^{a(s)}(\ol{\Omega})$, $L$ has a unique bounded extension $\bar{L} \in (H^{a(s)}(\ol{\Omega}))^*$. Consider next the homeomorphism in Lemma \ref{lemma_has_spaces}, 
\[
T = r_{\Omega}( (-\Delta)^a + q ): H^{a(s)}(\ol{\Omega}) \to H^{s-2a}(\Omega).
\]
Its adjoint is a bounded map between the dual Banach spaces, 
\[
T^*: (H^{s-2a}(\Omega))^* \to (H^{a(s)}(\ol{\Omega}))^*.
\]
The map $T^*$ is also a homeomorphism, with inverse given by $(T^{-1})^*$. Using the identification $(H^{s-2a}(\Omega))^* = H^{-s+2a}_{\ol{\Omega}}$ one has 
\[
T^* v(w) = (v, Tw), \qquad w \in H^{a(s)}(\ol{\Omega}).
\]

Now let $v \in H^{-s+2a}_{\ol{\Omega}}$ be the unique function satisfying $T^* v = \bar{L}$, and choose a sequence $(v_j)_{j=1}^{\infty} \subset C^{\infty}_c(\Omega)$ with $v_j \to v$ in $H^{-s+2a}$. If $f \in C^{\infty}_c(W)$, recall that $e^+ r_{\Omega} P_q f = P_q f - f$, and observe that 
\begin{align*}
0 &= L(e^+ r_{\Omega} P_q f) = \bar{L}(P_q f-f) = T^* v(P_q f - f) = (v, T(P_q f - f)) \\
 &=-(v,Tf) = -\lim \,(v_j, ((-\Delta)^a + q) f) = -\lim \,( ((-\Delta)^a + q) v_j, f ).
\end{align*}
Here we used that $T P_q f = 0$ and $v_j \in C^{\infty}_c(\Omega)$. Since $f \in C^{\infty}_c(W)$, we may take the limit as $j \to \infty$ and obtain that 
\[
((-\Delta)^a v, f) = 0, \qquad f \in C^{\infty}_c(W).
\]
Thus $v \in H^{-s+2a}(\mR^n)$ satisfies 
\[
v|_W = (-\Delta)^a v|_W = 0.
\]
By Theorem \ref{thm_uniqueness} it follows that $v \equiv 0$. This implies that $\bar{L} \equiv 0$ and $L \equiv 0$ as required.
\end{proof}

\begin{proof}[Proof of Theorem \ref{thm_approximation}]
Since $\mathrm{int}(\Omega_1 \setminus \Omega) = 0$, we may find a small ball $W$ with $\ol{W} \subset\Omega_1 \setminus \ol{\Omega}$. Part (a) is then a consequence of Lemma \ref{lemma_runge_fractional}. As for part (b), if $f \in C^{\infty}(\ol{\Omega})$ and if $g = e^+ d(x)^a f \in \mathcal{E}_a(\ol{\Omega})$, Lemma \ref{lemma_runge_fractional_higherorder} ensures that there is a sequence $(u_j)_{j=1}^{\infty} \subset H^s(\mR^n)$ with 
\[
((-\Delta)^s + q) u_j = 0 \text{ in $\Omega$}, \qquad \mathrm{supp}(u_j) \subset \ol{\Omega}_1,
\]
so that $e^+ r_{\Omega} u_j \in \mathcal{E}_a(\ol{\Omega})$ and 
\[
e^+ r_{\Omega} u_j \to g \text{ in $\mathcal{E}_a(\ol{\Omega})$.}
\]
The result will follow if we can show that 
\[
M: C^{\infty}(\ol{\Omega}) \to \mathcal{E}_a(\ol{\Omega}), \ \ Mf = e^+ d(x)^a f
\]
is a homeomorphism, since then applying $M^{-1} = d(x)^{-a} r_{\Omega}$ gives 
\[
d(x)^{-a} r_{\Omega} u_j \to f \text{ in $C^{\infty}(\ol{\Omega})$}.
\]
But $M$ is a bijective linear map between Fr\'echet spaces and has closed graph: if $f_j \to f$ in $C^{\infty}$ and $Mf_j \to h$ in $\mathcal{E}_a$, then also $Mf_j \to Mf$ in $L^{\infty}$ and one obtains $Mf = h$ by uniqueness of distributional limits. Thus $M$ is a homeomorphism by the closed graph and open mapping theorems (in other words, there is a unique Fr\'echet space topology on $\mathcal{E}_a(\ol{\Omega})$ stronger than the Hausdorff topology inherited from $\mathcal{D}'(\mR^n)$).
\end{proof}

\begin{Remark}
Let us note the following consequence of Theorem \ref{thm_approximation}(b): if $k \geq 0$ and $R > 1$ are fixed, then for any $g \in C^k(\ol{B}_1)$ and for any $\eps > 0$ there is a function $u \in H^s(\mR^n)$ satisfying 
\[
(-\Delta)^s u = 0 \text{ in } B_1, \quad \mathrm{supp}(u) \subset \ol{B}_R, \quad \norm{u-g}_{C^k(\ol{B}_1)} < \eps.
\]
This can be seen by taking $\Omega = B_r$ and $\Omega_1 = B_R$ where $1 < r < R$, and by choosing $f \in C^{\infty}(\ol{B}_r)$ with $\norm{f - d(x)^{-s} g}_{C^k(\ol{B}_1)}$ small enough.
\end{Remark}

\bibliographystyle{alpha}

\end{document}